\documentclass[12pt]{amsart}

\usepackage{amsfonts, amsthm, amsmath}

\usepackage{rotating}
\usepackage{tikz}

\usepackage{amscd}

\usepackage[latin2]{inputenc}

\usepackage{t1enc}
\usetikzlibrary{shapes,arrows}
\tikzstyle{pnt}=[draw,circle,fill,inner sep=1pt]
\tikzstyle{opnt}=[draw,circle,inner sep=2pt]

\usepackage[mathscr]{eucal}

\usepackage{indentfirst}

\usepackage{graphicx}

\usepackage{graphics}

\usepackage{pict2e}

\usepackage{mathrsfs}

\usepackage{enumerate}
\usepackage[pagebackref]{hyperref}
\hypersetup{backref, pagebackref, colorlinks=true}
\usepackage{cite}
\usepackage{color}
\usepackage{epic}
\usepackage{hyperref} 
\numberwithin{equation}{section}
\theoremstyle{plain}

\newtheorem{theorem}{Theorem}

\newtheorem{lemma}[theorem]{Lemma}

\newtheorem{prop}[theorem]{Proposition}
\newtheorem{cor}[theorem]{Corollary}
\newtheorem{remark}[theorem]{Remark}
\newtheorem{example}[theorem]{Example}

\newcommand{\Z}{\mathbb{Z}}
\def\rep{\mathrm{rep}}
\def\even{\mathrm{even}}
\def\dist{\mathrm{dist}}
\def\wt{\mathrm{wt}}
\def\mod{\mathrm{mod}}
\def\dif{\mathrm{dif}}
\def\pa{\mathrm{Par}}

\def\D{\mathfrak{D}}

\def\Par{\mathcal{P}}
\def\H{\mathcal{H}}
\def\la{\lambda}

\def\D{\mathcal{D}}
\def\M{\mathcal{M}}

\makeatother

\usepackage{xcolor}

\allowdisplaybreaks

\begin{document}

\title[Counting integer partitions by perimeter]{Combinatorics of integer partitions with\\ prescribed perimeter}

\author[Z. Lin]{Zhicong Lin}
\address[Zhicong Lin]{Research Center for Mathematics and Interdisciplinary Sciences, Shandong University, Qingdao 266237, P.R. China}
\email{linz@sdu.edu.cn}

\author[H. Xiong]{Huan Xiong}
\address[Huan Xiong]{Institute for Advanced Study in Mathematics,
 Harbin Institute of Technology, Harbin 150001, P.R. China}
 \email{huan.xiong.math@gmail.com}

\author[S.H.F. Yan]{Sherry H.F. Yan}
\address[Sherry H.F.  Yan]{Department of Mathematics,
Zhejiang Normal University, Jinhua 321004, P.R. China}
\email{hfy@zjnu.cn}

\date{\today}

\begin{abstract}
 We prove that the number of even parts and the number of times that parts are repeated have the same distribution over integer partitions with a fixed perimeter. This refines Straub's analog of Euler's Odd-Distinct partition theorem. We generalize the two concerned statistics to these of the part-difference less than $d$ and the parts not congruent to $1$ modulo $d+1$ and prove a distribution inequality, that has a similar flavor as Alder's ex-conjecture, over partitions with a prescribed perimeter. Both of our results are proved analytically and combinatorially.
\end{abstract}



\maketitle


\section{Introduction}

Integer partitions~\cite{AE} play important roles in combinatorics, number theory and  other related mathematical branches. For a positive integer $n$, a sequence $\la=(\la_1,\la_2,\ldots,\la_k)$ of weakly decreasing positive integers such that $\sum_i\la_i=n$ is called a {\em partition of $n$}. The $\la_i$'s are called the parts of $\la$ and $k$ is the number of parts that we denote $\ell(\la)$. The {\em perimeter} of $\la$ is defined to be
$$
\Gamma(\la)=\la_1+\ell(\la)-1
$$
 and the {\em minimal part-difference} of $\la$ is $$\min\{\la_{i}-\la_{i+1}: 1\leq i <\ell(\la)\}.$$

One of the most famous results in partition theory is  Euler's Odd-Distinct partition theorem, which asserts that the number of partitions of $n$ with odd parts  is equal to the number with distinct parts. Two beautiful bijective proofs (see~\cite[pp.~63-65]{St}) of Euler's partition theorem with different refinements were constructed respectively by Sylvester and  Glaisher. Recently, Straub~\cite[Thereom~1.4]{Strau} considered the set $\H_n$ of partitions with perimeter $n$, rather than the set of all partitions of $n$, and proved the following analog of Euler's partition theorem. 

\begin{theorem}[Straub]
\label{thm:stra}
There are as many partitions in $\H_n$ with odd parts as with distinct parts.
\end{theorem}

For a fixed positive integer $d$, Alder's ex-conjecture (see~\cite[Sec.~4.3]{AE} and~\cite{FT}) states that there are not more partitions of $n$ into parts congruent to $\pm1$ modulo $d+3$ than into parts with minimal part-difference at least $d$. Indeed,  Alder's ex-conjecture is Euler's partition theorem (resp.~Rogers--Ramanujan identity) for $d=1$ (resp.~$d=2$). 
Inspired by Alder's ex-conjecture, Fu and Tang~\cite[Theorem~2.15]{FT} proved the following $d$-extension of Straub's result. 

\begin{theorem}[Fu--Tang]
\label{thm:FT}
There are as many partitions in $\H_n$ with parts congruent to $1$ modulo $d+1$ as with minimal part-difference at least $d$.  
\end{theorem}

The purpose of this paper is to  study further  the combinatorics of integer partitions with a prescribed perimeter after the aforementioned works by Straub and Fu--Tang. 

We are interested in two natural statistics on partitions. Define for a partition $\la=(\la_1,\ldots,\la_k)$ the statistics 
$$
\rep(\la)= | \{ 1\leq i \leq \ell(\la)-1: \la_i = \la_{i+1} \} |\quad\text{and}\quad\even(\la)= | \{ 1\leq i \leq \ell(\la): \la_i \text{ is even} \} |,
$$
called the {\em number of  repeated parts} and the  {\em number of even parts}, respectively. 
Our first result refines Theorem~\ref{thm:stra} by the above two statistics. 
\begin{theorem}\label{th:multi_even}
For any positive integer $n$, we have
\begin{equation}\label{eq:multi_even}
\sum_{\la \in \H_n} t^{\rep(\la)} = \sum_{\la \in \H_n} t^{\even(\la)}.
\end{equation}
\end{theorem}

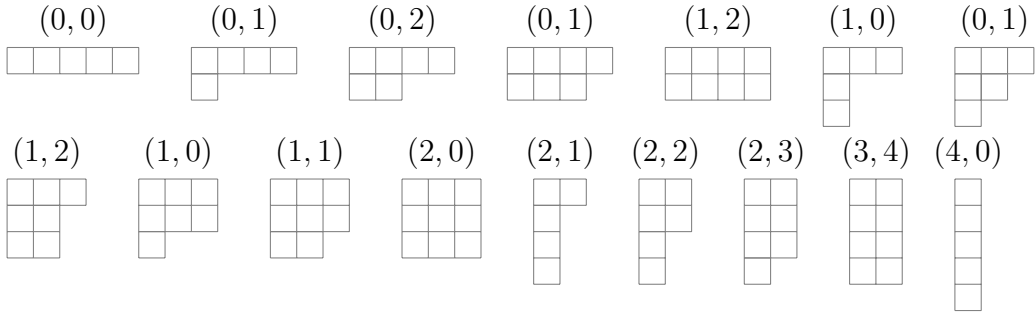
\begin{figure}
\begin{center}
\begin{tikzpicture}[scale=.35]

\draw(2.5,12) node{$(0,0)$};
\draw[step=1,color=gray] (0,10) grid (5,11);

\draw(9,12) node{$(0,1)$};
\draw[step=1,color=gray] (7,10) grid (11,11); 
\draw[step=1,color=gray] (7,9) grid (8,10); 

\draw(15,12) node{$(0,2)$};
\draw[step=1,color=gray] (13,10) grid (17,11); 
\draw[step=1,color=gray] (13,9) grid (15,10); 

\draw(21,12) node{$(0,1)$};
\draw[step=1,color=gray] (19,10) grid (23,11); 
\draw[step=1,color=gray] (19,9) grid (22,10); 

\draw(27,12) node{$(1,2)$};
\draw[step=1,color=gray] (25,9) grid (29,11); 

\draw(32.5,12) node{$(1,0)$};
\draw[step=1,color=gray] (31,10) grid (34,11); 
\draw[step=1,color=gray] (31,8) grid (32,10); 

\draw(37.5,12) node{$(0,1)$};
\draw[step=1,color=gray] (36,10) grid (39,11); 
\draw[step=1,color=gray] (36,8) grid (37,10); 
\draw[step=1,color=gray] (37,9) grid (38,10); 

\draw(1.5,7) node{$(1,2)$};
\draw[step=1,color=gray] (0,5) grid (3,6); 
\draw[step=1,color=gray] (0,3) grid (2,5); 

\draw(6.5,7) node{$(1,0)$};
\draw[step=1,color=gray] (5,4) grid (8,6); 
\draw[step=1,color=gray] (5,3) grid (6,4); 

\draw(11.5,7) node{$(1,1)$};
\draw[step=1,color=gray] (10,4) grid (13,6); 
\draw[step=1,color=gray] (10,3) grid (12,4); 

\draw(16.5,7) node{$(2,0)$};
\draw[step=1,color=gray] (15,3) grid (18,6); 

\draw(21,7) node{$(2,1)$};
\draw[step=1,color=gray] (20,5) grid (22,6); 
\draw[step=1,color=gray] (20,2) grid (21,5); 

\draw(25,7) node{$(2,2)$};
\draw[step=1,color=gray] (24,4) grid (26,6); 
\draw[step=1,color=gray] (24,2) grid (25,4); 

\draw(29,7) node{$(2,3)$};
\draw[step=1,color=gray] (28,3) grid (30,6); 
\draw[step=1,color=gray] (28,2) grid (29,3); 

\draw(33,7) node{$(3,4)$};
\draw[step=1,color=gray] (32,2) grid (34,6); 

\draw(36.5,7) node{$(4,0)$};
\draw[step=1,color=gray] (36,1) grid (37,6);

\end{tikzpicture}
\end{center}
\caption{Partitions in $\H_5$ (represented as Young diagrams) with their pair of statistics $(\rep,\even)$ on the top.\label{par:peri}}
\end{figure}

See Fig.~\ref{par:peri} for an example of~\eqref{eq:multi_even} for $n=5$, which shows
$$
\sum_{\la \in \H_5} t^{\rep(\la)} =5+5t+4t^2+t^3+t^4= \sum_{\la \in \H_5} t^{\even(\la)}.
$$ Setting $t=0$ in~\eqref{eq:multi_even} recovers Theorem~\ref{thm:stra}. We will investigate the combinatorics of Theorem~\ref{th:multi_even} by providing three different proofs: 
\begin{enumerate}
\item a recurrence relation proof with connection to extraordinary subsets studied by Grimaldi~\cite{Gri};
\item a generating function proof with some related consequences;
\item a bijective proof that admits an unexpected extension to general $d$ for our next result.
\end{enumerate}

Our next result can be considered as an extension of Theorem~\ref{thm:FT}, as will be seen (see Theorem~\ref{thm:bijd}) when exploiting  the $d$-extension of our bijective proof of Theorem~\ref{th:multi_even}.
\begin{theorem}\label{thm:ineq}
Fix integers $d\geq1$ and $n\geq1$. The  total number of indices $i$ with $1 \leq i< \ell(\lambda)$ satisfying $\lambda_i-\lambda_{i+1}<d$ in all partitions $\la$ in $\H_n$ is not less than  the total number of parts not congruent to $1$ modulo $d+1$  in all partitions in $\H_n$. 
\end{theorem}

Again, three different proofs for Theorem~\ref{thm:ineq} will be provided: a generating function proof, an injective proof and a bijective proof. 

The proofs of Theorems~\ref{th:multi_even} and~\ref{thm:ineq} are given in Sections~\ref{sec:2} and~\ref{sec:3}, respectively. This paper is concluded with further remarks in Section~\ref{sec:4}.

\section{Three proofs of Theorem~\ref{th:multi_even}}
\label{sec:2}

\subsection{A recurrence relation proof}
Let $A(n,k)$ (resp.~$B(n,k)$) be the number of partitions $\la$ in $\H_n$ with $\rep(\la) = k$ (resp.~$\even(\la) = k$). In the following two lemmas, we prove that $A(n,k)$ and $B(n,k)$ share the same (binomial-like) recurrence relation, which proves Theorem~\ref{th:multi_even}. 

\begin{lemma}The number $A(n,k)$ satisfies the recurrence relation
\begin{equation}\label{eq: A(n,k)}
   A(n,k)=A(n-1,k)+A(n-1,k-1)+A(n-2,k)-A(n-2,k-1)
\end{equation}
for $n\geq2$ with the initial value $A(1,0)=1$. 
\end{lemma}
\begin{proof}
Suppose that $n\geq2$ and  $\la = ( \la_1, \la_2, \la_3, \ldots)$ is a
  partition in $\H_n$ with $\rep(\la) = k$. To obtain the recursion formula for $A(n,k)$, we consider three cases:
\begin{enumerate}[(i)]
\item  $\la_1=\la_2 $. In this case, by deleting the largest part of $\la$, we obtain $\lambda' = (  \lambda_2,
    \lambda_3, \ldots)$. Then $\la'\in \H_{n-1}$ and $\rep(\la') = k-1$. Furthermore, $\la'$ can be any of
    the partitions satisfying $\la'\in \H_{n-1}$ and $\rep(\la') = k-1$. Therefore, the number of partitions $\la \in\H_n$ satisfying $\la_1=\la_2$ and $\rep(\la) = k$ is equal to $A(n-1,k-1)$.
    
\item  $\la_1=\la_2 +1 $. In this case, again by deleting the largest part of $\la$, we obtain $\lambda' = (  \lambda_2,
    \lambda_3, \ldots)$. It is clear that $\la'\in \H_{n-2}$ and $\rep(\la') = k$. Furthermore, $\la'$ can be any of the partitions satisfying $\la'\in \H_{n-2}$ and $\rep(\la') = k$. Therefore, the number of partitions $\la \in\H_n$ satisfying $\la_1=\la_2+1$ and $\rep(\la) = k$ is equal to $A(n-2,k)$.
    
\item $\la_1\geq \la_2 +2 $. In this case, by removing one box in the first row of the Young diagram of $\la$, we obtain $\lambda' = ( \la_1-1, \lambda_2,
\lambda_3, \ldots)$. It is clear that $\la'\in \H_{n-1}$,  $\rep(\la') = k$ and ${\lambda'}_1 \geq {\la'}_2 + 1$. By (i), the number of partitions $\la'$ satisfying $\la'\in \H_{n-1}$,  $\rep(\la') = k$ and ${\la'}_1 = {\la'}_2 $ is equal to $A(n-2,k-1)$. Thus, the number of partitions $\la'$ satisfying $\la'\in \H_{n-1}$,  $\rep(\la') = k$ and ${\la'}_1 \geq {\la'}_2 + 1$ is equal to $A(n-1,k)-A(n-2,k-1)$. Finally, this implies that the number of partitions $\la \in\H_n$ satisfying $\la_1\geq \la_2+2$ and $\rep(\la) = k$ is equal to $A(n-1,k)-A(n-2,k-1)$.
\end{enumerate}

Combining the above three cases gives~\eqref{eq: A(n,k)}.  
\end{proof}

\begin{lemma}The number $B(n,k)$ satisfies the recurrence relation
\begin{equation}\label{eq: B(n,k)}
   B(n,k)=B(n-1,k)+B(n-1,k-1)+B(n-2,k)-B(n-2,k-1)
\end{equation}
for $n\geq2$ with the initial value $B(1,0)=1$. 
\end{lemma}

\begin{proof}
Suppose that $n\geq 2$ and  $\la = ( \la_1, \la_2, \la_3, \ldots)$ is a
  partition in $\H_n$ with $\even(\la) = k$. Let $B_o(n,k)$ (resp.~$B_e(n,k)$)  be the number of partitions $\la\in\H_n$ with $\la_1$ odd (resp.~even) and $\even(\la) = k$. To obtain the recursion formula for $B(n,k)$ is more involved and we again consider three cases:
\begin{enumerate}[(i)]
\item $\la_1=\la_2 $. In this case, consider the  deletion of the largest part of $\la$, we see that the number of partitions $\la\in\H_n$ satisfying $\la_1$ odd (resp.~even) and $\even(\la) = k$ is equal to $B_o(n-1,k)$  (resp.~$B_e(n-1,k-1)$).
    
\item $\la_1=\la_2 +1 $. In this case, again by considering  the  deletion of the largest part of $\la$, we see that the number of partitions $\la\in\H_n$ satisfying $\la_1$ odd (resp.~even), $\la_1=\la_2+1$ and $\even(\la) = k$ is equal to $B_e(n-2,k)$  (resp.~$B_o(n-2,k-1)$).
    
\item $\la_1\geq \la_2 +2 $. In this case, by removing two boxes in the first row of the Young diagram of $\la$, we obtain $\lambda' = ( \la_1-2, \lambda_2,
\lambda_3, \ldots)$. It is clear that $\la'\in \H_{n-2}$ and  $\even(\la') = k$. Furthermore, $\la'$ can be any of
    the partitions satisfying $\la'\in \H_{n-2}$ and $\even(\la') = k$. Therefore, the number of partitions $\la \in\H_n$ satisfying $\la_1$ odd (resp.~even), $\la_1\geq \la_2+2$ and $\even(\la) = k$ is equal to $B_o(n-2,k)$ (resp.~$B_e(n-2,k)$).
   \end{enumerate} 

Combining  the above three cases together, we derive that 
\begin{equation} \label{eq: B'(n,k)}
    B(n,k)=B(n-2,k)+B_o(n-1,k)+B_o(n-2,k-1)+B_e(n-1,k-1)+B_e(n-2,k).
\end{equation}
On the other hand, by deleting the first part of $\la$ when $\la_1=\la_2$ or by removing one box in the first row of the Young diagram of $\la$ when $\la_1\geq \la_2 +1 $, we see that  
\begin{equation} \label{eq: B_1}
B_o(n-1,k-1)=B_o(n-2,k-1)+B_e(n-2,k)
\end{equation} 
and
\begin{equation} \label{eq: B_2}
B_e(n-1,k)=B_e(n-2,k-1)+B_o(n-2,k-1)=B(n-2,k-1).
\end{equation} 
Now~\eqref{eq: B'(n,k)},~\eqref{eq: B_1} and~\eqref{eq: B_2} together results in~\eqref{eq: B(n,k)}, as desired.   
\end{proof}

For positive integers $n$ and $k$ with $1\leq k\leq n$, a subset $S$ of $[n]:=\{1,2,\ldots,n\}$ is called a {\em$k$-extraordinary subset} if $|S|$ equals the $k$-th smallest element of $S$. For example, all the $2$-extraordinary subsets of $[5]$ are 
$$
\{1,2\}, \,\,\{1,3,4\},\,\, \{1,3,5\}, \,\,\{2,3,4\},\,\, \{2,3,5\}. 
$$
Let $C(n,k)$ be the number of $k$-extraordinary subsets of $[n]$. In~\cite{Gri}, Grimaldi showed that 
$$
C(n,k)=C(n-1,k)+\sum_{i=1}^kC(n-k-2+i,i)
$$
for $n\geq2$ and the initial value $C(1,1)=1$. 
It follows from the above recursion that 
\begin{align*}
C(n,k)&=C(n-1,k)+C(n-2,k)+\sum_{i=1}^{k-1}C(n-k-2+i,i)\\
&=C(n-1,k)+C(n-2,k)+C(n-1,k-1)-C(n-2,k-1). 
\end{align*}
Comparing with the recursions for $A(n,k)$ and $B(n,k)$ derived in the above two lemmas (and their initial values) yields the following result. 
\begin{prop}
For $n\geq1$ and $0\leq k\leq n-1$, we have 
$$
A(n,k)=B(n,k)=C(n,k+1). 
$$
\end{prop}

  \subsection{A generating function proof}
  We will compute the generating function for the joint distribution of $(\rep,\even)$ on $\H_n$. 
  For a partition $\lambda$, let $\dist(\lambda)$  be the number of {\em distinct parts}  of $\lambda$. 
  Then, $\dist(\la)=\ell(\la)-\rep(\la)$. 
  A partition $\la$ can also be represented as $\lambda=1^{m_1}2^{m_2}\cdots$, where $m_i$ is the number of parts equating $i$ of $\la$.  Consider the weight function 
 $$
 \wt(\lambda)=p^{\dist(\lambda)}q^{\even(\lambda)}x_1^{m_1}x_2^{m_2}\cdots.
 $$
 Then, 
  \begin{equation}\label{gf:par}
1+\sum_{\la\in\Par} \wt(\lambda)=\biggl(1+\frac{px_1}{1-x_1}\biggr)\biggl(1+\frac{pqx_2}{1-qx_2}\biggr)\biggl(1+\frac{px_3}{1-x_3}\biggr)\biggl(1+\frac{pqx_4}{1-qx_4}\biggr)\cdots,
\end{equation}
where $\Par$ denotes the set of  all integer partitions. 
  Let $\Par_n$ be the set of partitions $\lambda=(\lambda_1,\lambda_2,\ldots)$ such that $\lambda_1=n$. Then, 
 \begin{equation}\label{gf:opar}
\sum_{\lambda\in\Par_{2n}} \wt(\lambda)=\biggl(1+\frac{px_1}{1-x_1}\biggr)\biggl(1+\frac{pqx_2}{1-qx_2}\biggr)\cdots\biggl(1+\frac{px_{2n-1}}{1-x_{2n-1}}\biggr)\biggl(\frac{pqx_{2n}}{1-qx_{2n}}\biggr)
\end{equation}
and 
 \begin{equation}\label{gf:epar}
\sum_{\lambda\in\Par_{2n+1}} \wt(\lambda)=\biggl(1+\frac{px_1}{1-x_1}\biggr)\biggl(1+\frac{pqx_2}{1-qx_2}\biggr)\cdots\biggl(1+\frac{pqx_{2n}}{1-qx_{2n}}\biggr)\biggl(\frac{px_{2n+1}}{1-x_{2n+1}}\biggr).
\end{equation}
It follows from~\eqref{gf:opar} and~\eqref{gf:epar} that 
\begin{align*}
&\quad\sum_{\lambda\in\Par}p^{\dist(\lambda)}q^{\even(\lambda)}x^{\ell(\lambda)}y^{\lambda_1}\\
&=\frac{pxy}{1-x}+\biggl(\frac{pqx}{1-qx}+\frac{pxy(1-qx+pqx)}{(1-x)(1-qx)}\biggr)\sum_{n\geq1}\biggl(1+\frac{px}{1-x}\biggr)^n\biggl(1+\frac{pqx}{1-qx}\biggr)^{n-1}y^{2n}\\
&=\frac{pxy}{1-x}+\frac{\biggl(\frac{pqx}{1-qx}+\frac{pxy(1-qx+pqx)}{(1-x)(1-qx)}\biggr)\biggl(1+\frac{px}{1-x}\biggr)y^2}{1-\biggl(1+\frac{px}{1-x}\biggr)\biggl(1+\frac{pqx}{1-qx}\biggr)y^2}.
\end{align*}
Therefore, we have 
\begin{theorem} \label{gf:rep-even}
The generating function for partitions by size of first parts, number of parts and the pair $(\dist,\even)$ is 
\begin{equation}\label{eq:main}
\sum_{\la\in\Par}p^{\dist(\lambda)}q^{\even(\lambda)}x^{\ell(\lambda)}y^{\lambda_1}=\frac{pxy}{1-x}+\frac{\biggl(\frac{pqx}{1-qx}+\frac{pxy(1-qx+pqx)}{(1-x)(1-qx)}\biggr)\biggl(1+\frac{px}{1-x}\biggr)y^2}{1-\biggl(1+\frac{px}{1-x}\biggr)\biggl(1+\frac{pqx}{1-qx}\biggr)y^2}. 
\end{equation}
Consequently ($p\leftarrow p^{-1}, x\leftarrow px, y\leftarrow x$),
\begin{equation*}
\sum_{\lambda\in\Par}p^{\rep(\lambda)}q^{\even(\lambda)}x^{\Gamma(\lambda)}=\frac{x(1-(p-1)q(x^2+x))}{1-p(1+q)x-(1-p^2q)x^2-(1-p)(1+q)x^3-(p-1)^2qx^4}.
\end{equation*}
In particular, 
\begin{equation}
\sum_{\lambda\in\Par}p^{\rep(\lambda)}x^{\Gamma(\lambda)}=\sum_{\lambda\in\Par}p^{\even(\lambda)}x^{\Gamma(\lambda)}=\frac{x}{1-(1+p)x-(1-p)x^2}.
\end{equation}
\end{theorem}

\begin{remark}
Dividing both sides of~\eqref{eq:main} by $x$ and then substituting $p\leftarrow p^{-1}, q\leftarrow 1, x\leftarrow -px, y\leftarrow x$ yields
\begin{equation}
\sum_{\lambda\in\Par}(-1)^{\ell(\lambda)}p^{\rep(\lambda)}x^{\Gamma(\lambda)}=\frac{x}{(p-1)(x^2-x)-1}.
\end{equation}
Thus, if we denote $$h_n(p)=\sum_{\Gamma(\lambda)=n}(-1)^{\ell(\lambda)}p^{\rep(\lambda)}=\sum_{\Gamma(\lambda)=n\atop \ell(\lambda)\text{ even}}p^{\rep(\lambda)}-\sum_{\Gamma(\lambda)=n\atop \ell(\lambda)\text{ odd}}p^{\rep(\lambda)},$$
then $\sum_{n\geq1}h_n(p)x^n=\frac{x}{(p-1)(x^2-x)-1}$,  which is equivalent to the recurrence relation
\begin{equation}
h_n(p)=(1-p)(h_{n-1}(p)-h_{n-2}(p))
\end{equation}
for $n\geq3$ with initial values $h_1(p)=-1$ and $h_2(p)=p-1$. In particular, we have 
$$
h_n(0)=
\begin{cases}
(-1)^m,\quad&\text{if $n=(6m-3\pm1)/2$};\\
0, &\text{otherwise}.
\end{cases}
$$
This was first proved by Fu and Tang~\cite[Theroem~1.4]{FT}, which is an analog to Euler's pentagonal number theorem (see~\cite[Sec.~3.5]{AE}). Two other special evaluations of $h_n(p)$ deserve to be mentioned:
\begin{itemize}
\item We have $h_n(1)=0$, which means that there are as many partitions in $\H_n$ with odd number of parts as with even number of parts. It is an interesting exercise to construct an involution proof of this simple fact. 
\item $|h_n(2)|$ is equal to the $n$-th Fibonacci number $F_n$, which satisfies $F_n=F_{n-1}+F_{n-2}$ for $n\geq3$ and initial values $F_1=F_2=1$. 
\end{itemize}
\end{remark}

Let $a_{o}$ (resp.~$a_{e}$) be the total number of odd (resp.~even) parts in all partitions in $\H_n$.
It can be easily deduced from~\eqref{eq:main} the following closed formulae for $a_{o}$ and $a_e$.
\begin{cor}\label{thm:diff_even_odd}
For all positive integer $n\geq 2$, we have
\begin{equation*}\label{eq:ana-pent}
a_{o} = (n+2) \cdot 2^{n-3} \qquad \text{and} \qquad a_{e} = n\cdot 2^{n-3}.
\end{equation*}
The above result means that, the average difference between the number of odd parts and the number of even parts in all partitions with perimeter $n$ is equal to $1/2$. 
\end{cor}

\subsection{A bijective proof}

This subsection is devoted to  a bijective proof of the equidistribution on $\H_n$:
$$
\sum_{\lambda\in\H_n}t^{\rep(\la)}=\sum_{\lambda\in\H_n}t^{\even(\la)}. 
$$

\begin{theorem}\label{thm:bij}
There exists a recursively defined bijection $\phi: \H_n\rightarrow \H_n$ satisfying  
$$\rep(\la)=\even(\phi(\la))$$ for every $\la\in\H_n$. 
\end{theorem}

It is convenience to use the following   representation of partitions as $01$-sequences. 
\begin{lemma}\label{lem:01sequence}
There exists a natural bijection  $\la\mapsto w(\la)=w_0w_1w_2\cdots w_n$ between partitions in $\H(n)$ and the set $\Lambda_n$ of  $01$-sequences  such that $w_0=0,w_n=1$. Furthermore, 
$$
\ell(\la) = |\{ 1\leq i \leq n: w_i  = 1 \}|,
$$
$$
\rep(\la) = |\{ 1\leq i \leq n: w_i = w_{i-1} = 1 \}|
$$
and
$$
\even(\la) = |\{ 1\leq i \leq n: w_i  = 1\text{ and }\#\{0\leq j<i: w_j = 0\} \text{ is even}    \}|.
$$
\end{lemma}

\begin{figure}
\begin{center}
\begin{tikzpicture}[scale=.75]

\draw(0.5,-0.3) node{$0$};
\draw[step=1,color=gray] (0,0) grid (1,1);
\draw[step=1,color=gray] (0,1) grid (3,3);
\draw[step=1,color=gray] (0,3) grid (6,4);
\draw(1.6,0.7) node{$0$};\draw(2.5,0.7) node{$0$};
\draw(1.2,0.5) node{$1$};\draw(3.2,1.5) node{$1$};\draw(3.2,2.5) node{$1$};\draw(6.2,3.5) node{$1$};
\draw(3.6,2.7) node{$0$};\draw(4.5,2.7) node{$0$};\draw(5.5,2.7) node{$0$};

\draw [very thick]
(0,0)--(1,0)--(1,1)--(3,1)--(3,3)--(6,3)--(6,4);
\end{tikzpicture}
\end{center}
\caption{The 01-sequence of $\lambda = (6,3,3,1)$ is $w(\la)=0100110001$.\label{01rep}}
\end{figure}
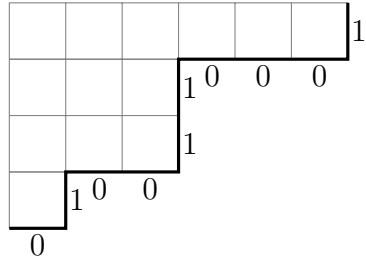

\begin{proof}
Visually, the $01$-sequence representation $w(\la)$ of a partition $\lambda\in\H_n$ can be obtained as follows.
For the edges in the  boundary of the Young diagram of $\lambda$, starting at the left bottom and ending to the right top, we label the vertical
(resp.~horizontal) edges with 1 (resp.~0); see Fig.~\ref{01rep} for an example. In this way, we get a 01-sequence $w(\lambda)\in\Lambda_n$. The three desired properties are obvious from this correspondence. 
\end{proof}

\begin{proof}[{\bf Proof of Theorem~\ref{thm:bij}}]
For each  $w\in\Lambda_n$, denote 
\begin{align*}
\widetilde\rep(w) &= |\{ 1\leq i \leq n: w_i = w_{i-1} = 1 \}|,\\
\widetilde\even(w) &= |\{ 1\leq i \leq n: w_i  = 1\text{ and }\#\{0\leq j<i: w_j = 0\} \text{ is even}   \}|.
\end{align*}
In view of Lemma~\ref{lem:01sequence},  we aim to define recursively a bijection $\tilde\phi: \Lambda_n\rightarrow \Lambda_n$ satisfying 
\begin{equation}\label{eq:stat}
\widetilde\rep(w)=\widetilde\even(\tilde\phi(w))
\end{equation}
for each $w\in\Lambda_n$ and then set $\phi=w^{-1}\circ\tilde\phi\circ w$. Set $\tilde\phi(\emptyset)=\emptyset$ and $\tilde\phi(01)=01$. For each $w\in\Lambda_n$ ($n\geq2$), we distinguish the following two cases:
\begin{itemize}
\item If $w=00w_2\cdots w_n$, then $\tilde\phi(w)$ is obtained from $\tilde\phi(0w_2\cdots w_n)$ by inserting  $1$ immediately after the initial $0$.
\item   If $w=01w_2\cdots w_n$, then suppose $w_m=0$ is the second $0$ in $w$ if exists, otherwise set $m=n+1$. Define $\tilde\phi(w)$ to be the concatenation  $00w_2\cdots w_{m-1}\tilde\phi(w_m\cdots w_n)$.
\end{itemize}
For example, we have
\begin{align*}
&001\mapsto011\mapsto001;\,\, 0101\mapsto 0001\mapsto 0111\mapsto 0011\mapsto 0101;\\
&00001\mapsto01111\mapsto00111\mapsto01011\mapsto00001,\, 01001\mapsto00011\mapsto01101\mapsto00101\mapsto01001
\end{align*}
under the mapping $\tilde\phi$.
It is routine to check by induction on $n$ that $\tilde\phi$ is well-defined and  satisfies~\eqref{eq:stat}. To see that $\tilde\phi$ is a bijection, we define its inverse explicitly. Given $w\in\Lambda_n$ ($n\geq2$), we consider the following two cases:
\begin{itemize}
\item If $w=01w_2\cdots w_n$, then $\tilde\phi^{-1}(w)$ is obtained from $\tilde\phi(0w_2\cdots w_n)$ by inserting  $0$ immediately after the initial $0$.
\item   If $w=00w_2\cdots w_n$, then suppose $w_m=0$ is the third $0$ in $w$ if exists, otherwise set $m=n+1$. Define $\tilde\phi^{-1}(w)$ to be the concatenation  $01w_2\cdots w_{m-1}\tilde\phi^{-1}(w_m\cdots w_n)$.
\end{itemize}
It is easy to check by induction on $n$ that $\tilde\phi$ and $\tilde\phi^{-1}$ are inverse to each other. 
\end{proof}

\section{Three proofs of Theorem~\ref{thm:ineq}}
\label{sec:3}

For a partition $\la$, let 
$$
\mod_d(\lambda):=|\{i: \la_i\equiv1 \text{($\mod\,d+1$)}\}|\quad\text{and}\quad\mod'_d(\lambda):=\ell(\la)-\mod_d(\lambda).
$$
Then, $\mod'_d(\lambda)$ is  $d$-generalization of $\even(\la)$ as $\mod'_1(\lambda)=\even(\la)$. Introduce the $d$-generalization of $\rep(\la)$ by 
$$
\dif_d(\la):=|\{i: 1\leq i<\ell(\la), \la_i-\la_{i+1}<d\}|.
$$
It is clear that Theorem~\ref{thm:ineq} is equivalent to
\begin{equation}\label{eq:ineq}
\sum_{\la\in\H_n}\dif_d(\la)\geq \sum_{\la\in\H_n}\mod'_d(\la).
\end{equation}

\subsection{A generating function proof}

First we compute the generating function for partitions by the perimeter and the statistic $\mod_d$. 
\begin{lemma}
\label{le:gfm}
We have
\begin{equation}\label{gf:mod1}
\sum_{\lambda\in\Par}t^{\mod_d(\lambda)}x^{\ell(\lambda)}y^{\lambda_1}=\frac{txy(1-x)^d+\frac{xy}{x+y-1}(y^{d+1}-y(1-x)^d)}{(1-tx)(1-x)^d-y^{d+1}}.
\end{equation}
In particular,
\begin{equation}\label{gf:nmod1}
\sum_{\lambda\in\Par}t^{\mod_d'(\lambda)}x^{\Gamma(\la)}=\frac{x(tx^{d+1}+(1-tx)^d(x-1))}{(1-x-tx)((1-tx)^d(x-1)+x^{d+1})}.
\end{equation}
\end{lemma}

\begin{proof}
Since every integer partition $\lambda$ can be written as $\lambda=1^{m_1}2^{m_2}\cdots$, we have
 \begin{align*}
\sum_{\lambda\in\Par} t^{m_d(\lambda)}x^{\ell(\lambda)}y^{\lambda_1}&=tx\sum_{n\equiv1(\mod\,d+1)}y^n\biggl(\frac{1}{1-tx}\biggr)^{\lceil n/(d+1)\rceil}\biggl(\frac{1}{1-x}\biggr)^{n-\lceil n/(d+1)\rceil},\\
&\qquad+x\sum_{n\not\equiv1(\mod\,d+1)}y^n\biggl(\frac{1}{1-tx}\biggr)^{\lceil n/(d+1)\rceil}\biggl(\frac{1}{1-x}\biggr)^{n-\lceil n/(d+1)\rceil}\\
&=tx\sum_{k\geq0}\biggl(\frac{y}{1-tx}\biggr)^{k+1}\biggl(\frac{y}{1-x}\biggr)^{dk}\\
&\qquad+x\sum_{i=1}^d\sum_{k\geq0}\biggl(\frac{y}{1-tx}\biggr)^{k+1}\biggl(\frac{y}{1-x}\biggr)^{dk+i},
\end{align*}
which is simplified to~\eqref{gf:mod1}.  
\end{proof}

Next  we compute the generating function for partitions by the perimeter and the statistic $\dif_d$. 
\begin{lemma}\label{le:gfd}
We have
\begin{equation}
  \sum_{\lambda\in\Par}t^{\dif_d(\la)}x^{\ell(\lambda)}y^{\lambda_1} =  \frac{xy}{1-y-tx(1-y^d)-xy^d}.
\end{equation}
In particular, 
\begin{equation}\label{gf:dif}
  \sum_{\lambda\in\Par}t^{\dif_d(\la)}x^{\Gamma(\lambda)} =  \frac{x}{1-x-tx(1-x^d)-x^{d+1}}.
\end{equation}
\end{lemma} 
\begin{proof}
Notice that in the 01-sequence $w(\la)$ of $\la$, each maximal segment $00\cdots01$ ($k$ 0's followed by one 1) after the leftmost $1$ contributes a $txy^{k}$ in the following generating function if $0\leq k \leq d-1$; and  contributes a $xy^{k}$ if $k \geq d$. 
Therefore, we have 
\begin{align*}
  \sum_{\lambda\in\Par}t^{\dif_d(\la)}x^{\ell(\lambda)}y^{\lambda_1} &= \frac{xy}{1-y} \cdot \left(1+ \biggl(    \frac{tx(1-y^d)}{1-y} +  \frac{xy^d}{1-y}      \biggr) + \biggl(    \frac{tx(1-y^d)}{1-y} +  \frac{xy^d}{1-y} \biggr)^2 + \cdots  \right)
  \\& = \frac{xy}{1-y-tx(1-y^d)-xy^d},
\end{align*}
as desired.
\end{proof}

Take the derivative with respect to $t$ and then set $t=1$ in~\eqref{gf:nmod1} gives 
\begin{equation}\label{eq:nmod1}
\sum_{n\geq1}x^n\sum_{\la\in\H_n}\mod'_d(\la)=\frac{x(1-x)(x^{d+1}-x(1-x)^d)}{(1-2x)^2(x^{d+1}-(1-x)^{d+1})}.
\end{equation}
On the other hand, the same operation on~\eqref{gf:dif} yields 
\begin{equation}\label{eq:difd1}
\sum_{n\geq1}x^n\sum_{\la\in\H_n}\dif_d(\la)=\frac{x^2(1-x^d)}{(1-2x)^2}.
\end{equation}
To finish the proof of Theorem~\ref{thm:ineq} (or equivalently, inequality~\eqref{eq:ineq}), it remains to show that 
\begin{equation}
\frac{x^2(1-x^d)}{(1-2x)^2}-\frac{x(1-x)(x^{d+1}-x(1-x)^d)}{(1-2x)^2(x^{d+1}-(1-x)^{d+1})}=\frac{x^{d+2}(1-2x+x^{d+1}-(1-x)^{d+1})}{(1-2x)^2((1-x)^{d+1}-x^{d+1})}
\end{equation}
has nonnegative coefficients for each $d\geq1$. Since 
$$
\frac{1-2x+x^{d+1}-(1-x)^{d+1}}{(1-2x)^2((1-x)^{d+1}-x^{d+1})}=\frac{1}{1-2x}\biggl(\frac{1}{(1-x)^{d+1}-x^{d+1}}-\frac{1}{1-2x}\biggr),
$$
Theorem~\ref{thm:ineq} then follows from the following interesting positivity result.

\begin{theorem}\label{thm:ration}
The rational function 
$$
\frac{1}{(1-x)^{d+1}-x^{d+1}}-\frac{1}{1-2x}=:\Delta_d(x)
$$
has nonnegative coefficients for each $d\geq1$.
\end{theorem}

Note that $\Delta_0(x)=\Delta_1(x)=0$ and 
\begin{align*}
\Delta_2(x)&=\sum_{n\geq1}x^n\sum_{i\geq0}{n\choose 3i+1}=x+2x^2+3x^3+5x^4+10x^5+21x^6+\cdots,\\
\Delta_3(x)&=\sum_{n\geq1}x^n\sum_{i\geq0}2{n+1\choose 4i+2}=2x+6x^2+12x^3+20x^4+32x^5+56x^6+\cdots,\\
\Delta_4(x)&=\sum_{n\geq1}x^n\sum_{i\geq0}\biggl(3{n+2\choose 5i+3}-{n\choose 5i+2}\biggr)=3x+11x^2+27x^3+54x^4+95x^5+156x^6+\cdots. 
\end{align*}
For two polynomials $f(x),g(x)\in\Z[x]$, we write 
$$
f(x)\geq_x g(x)\Leftrightarrow f(x)-g(x)\text{ has nonnegative coefficients}.
$$
To prove Theorem~\ref{thm:ration}, we need the following auxiliary lemma.

\begin{lemma}\label{lem:bino}
For $d\geq1$, we have 
$$
\frac{x}{(1-x)^d}\geq_x\frac{x^{d}}{(1-x)^d}.
$$
\end{lemma}
\begin{proof}
This follows from the binomial theorem
$$
(1-x)^{-d}=\sum_{k\geq0}{d+k-1\choose k}x^k
$$
and the monotonicity ${d+k\choose k+1}\geq{d+k-1\choose k}$ for $k\geq0$. 
\end{proof}
We can now prove Theorem~\ref{thm:ration}. 

\begin{proof}[{\bf Proof of Theorem~\ref{thm:ration}}]
By Lemma~\ref{lem:bino} and by induction on $d$, we have 
\begin{align*}
&\quad\frac{1}{(1-x)^{d}}-\frac{1}{1-x}\biggl(1+\frac{x}{1-x}+(\frac{x}{1-x})^2+\cdots+(\frac{x}{1-x})^{d-1}\biggr)\\
&\geq_x \frac{1}{(1-x)^{d}}-\frac{x}{(1-x)^{d}}-\frac{1}{1-x}\biggl(1+\frac{x}{1-x}+(\frac{x}{1-x})^2+\cdots+(\frac{x}{1-x})^{d-2}\biggr)\\
&=\frac{1}{(1-x)^{d-1}}-\frac{1}{1-x}\biggl(1+\frac{x}{1-x}+(\frac{x}{1-x})^2+\cdots+(\frac{x}{1-x})^{d-2}\biggr)\\
&\geq_x0
\end{align*}
for each $d\geq 2$. It then follows from the above assertion that 
\begin{align*}
\Delta_{d-1}(x)&=\frac{1}{(1-x)^{d}-x^{d}}-\frac{1}{1-2x}\\
&=\frac{1}{(1-x)^{d}}\frac{1}{1-(\frac{x}{1-x})^{d}}-\frac{1}{1-x}\frac{1}{1-\frac{x}{1-x}}\\
&=\frac{1}{(1-x)^{d}}\biggl(1+(\frac{x}{1-x})^d+(\frac{x}{1-x})^{2d}+(\frac{x}{1-x})^{3d}+\cdots\biggr)\\
&\quad-\frac{1}{1-x}\biggl(1+\frac{x}{1-x}+(\frac{x}{1-x})^{2}+(\frac{x}{1-x})^{3}+\cdots\biggr)\\
&=\biggl(\sum_{k\geq 0}(\frac{x}{1-x})^{kd}\biggr)\biggl(\frac{1}{(1-x)^{d}}-\frac{1}{1-x}\biggl(1+\frac{x}{1-x}+(\frac{x}{1-x})^2+\cdots+(\frac{x}{1-x})^{d-1}\biggr)\biggr)\\
&\geq_x0,
\end{align*}
which completes the proof of the theorem. 
\end{proof}

\subsection{An injective proof}
This section is motivated by finding an interpretation for 
$$
\sum_{\la\in\H_n}\dif_d(\la)-\sum_{\la\in\H_n}\mod'_d(\la).
$$
A partition $\lambda=(\la_1,\la_2,\cdots,\la_i^*,\cdots)$ whose $i$-th  part, $1\leq i\leq\ell(\la)$,  receives a star is called a {\em labeled partition}. It is convenience to represent such a labeled partition by $(\la,i)$. Let us consider the two sets of labeled partitions:
\begin{align*}
&\D_{n,d}:=\{(\la,i): \la\in\H_n, 0\leq \la_{i-1}-\la_i<d\}\quad\text{and }\\
 &\M_{n,d}:=\{(\la,i): \la\in\H_n, \la_i\not\equiv1(\mod\,d+1)\}.
\end{align*}
For example, 
$$
\D_{4,2}=\{32^*,33^*,211^*,21^*1,221^*,22^*1,222^*,22^*2,1111^*,111^*1,11^*11\}
$$
and 
$$
\M_{4,2}=\{3^*1,32^*,3^*2,33^*,3^*3,2^*11,22^*1,2^*21,222^*,22^*2,2^*22\}.
$$
It is clear that Theorem~\ref{thm:ineq} is equivalent to $|\D_{n,d}|\geq|\M_{n,d}|$ for $n,d\geq1$. The purpose of this section is to provide an injective proof of Theorem~\ref{thm:ineq} that leads to a partition interpretation of $|\D_{n,d}|-|\M_{n,d}|$. 

\begin{theorem}
For any fixed $n,d\geq1$, there exists an injection $\xi$ from $\M_{n,d}$ to $\D_{n,d}$ such that 
\begin{equation}\label{p:inject}
\D_{n,d}\setminus\xi(\M_{n,d})=\{(\la,i)\in\D_{n,d}: \la_{i-1}\equiv1(\mod\,d+1), \la_i\not\equiv1(\mod\,d+1)\}.
\end{equation}
\end{theorem}

\begin{proof}
For a labeled partition $(\la,i)\in\M_{n,d}$, we define $\xi(\la,i)\in\D_{n,d}$ according to the following two cases. 
Suppose that $\la_i=l(d+1)+k$ for some $2\leq k\leq d+1$ and $l\geq0$. Note that $0\leq k-2<d$ and we use the convention $\la_{\ell(\la)+1}=0$.
\begin{itemize}
\item If $\la_i-\la_{i+1}\leq k-2$, then set $\xi(\la,i)=(\la,i+1)$.   As $0\leq k-2<d$, it is clear that $\xi(\la,i)\in\D_{n,d}$. In this case, $\la_{i+1}\not\equiv1(\mod\,d+1)$.
\item If $\la_i-\la_{i+1}> k-2$, then set $\xi(\la,i)=(\la',i+1)$, where $\la'=\la'_1\la'_2\cdots\la'_{\ell(\la)+1}$ is a  partition with one more part than $\la$ defined as 
$$
\la'_j=\begin{cases}
\la_j-1,&\quad\text{if $j\leq i$};\\
l(d+1)+1, &\quad\text{if $j=i+1$};\\
\la_{j-1}, &\quad\text{if $i+1<j\leq\ell(\la)+1$}.
\end{cases}
$$
Since $\la'_i-\la'_{i+1}=l(d+1)+k-1-l(d+1)-1=k-2$, we see $\xi(\la,i)\in\D_{n,d}$. 
In this case, $\la_{i+1}=l(d+1)+1\equiv1(\mod\,d+1)$. 
\end{itemize}
For example, if $(\la,3)=(14,13,11^*,10,5,2)\in\M_{19,5}$ then $\xi(\la,3)=(14,13,11,10^*,5,2)\in\D_{19,5}$ is constructed in the first case, and if $(\la,3)=(14,13,11^*,5,5,2)\in\M_{19,5}$ then $\xi(\la,3)=(13,12,10,7^*,5,5,2)\in\D_{19,5}$ is constructed in the second case. 

To see that the mapping $\xi$ is an injection, observe that for any $(\la,i+1)\in\D_{n,d}$:
\begin{itemize}
\item $(\la,i+1)$ is the image under $\xi$ from the first case above if  $\la_{i+1}\not\equiv1(\mod\,d+1)$ and $\la_{i}\not\equiv1(\mod\,d+1)$;
\item $(\la,i+1)$ is the image under $\xi$ from the second case above if  $\la_{i+1}\equiv1(\mod\,d+1)$;
\item $(\la,i+1)$ is not an image under $\xi$ if $\la_{i+1}\not\equiv1(\mod\,d+1)$ but $\la_{i}\equiv1(\mod\,d+1)$.
\end{itemize}
Since the above two cases of $\xi$ are reversible, $\xi$ is an injection and~\eqref{p:inject} holds. 
\end{proof}

\begin{example}As an example of~\eqref{p:inject}, we see that the number $|\D_{6,2}|-|\M_{6,2}|=4$ counts the labeled partitions $(4,3^*,1)$, $(4,3^*,2)$, $(4,4,3^*)$ and $(4,3^*,3)$.
\end{example}

\subsection{A bijective proof}
The following stronger version of Theorem~\ref{thm:ineq}  was originally suggested by numerical computations.
\begin{theorem}\label{thm:bijd}
Fix $d\geq1$. There exists a  bijection $\phi_d: \H_n\rightarrow \H_n$ such that 
\begin{equation*}
\dif_d(\la)\geq\mod'_d(\phi_d(\la))
\end{equation*}
 for every $\la\in\H_n$. Moreover, $\dif_d(\la)=0$ if and only if $\mod'_d(\phi_d(\la))=0$. 
\end{theorem}

\begin{remark}
Theorem~\ref{thm:bijd} is a common generalization of Theorems~\ref{thm:FT} and~\ref{thm:ineq}. At the beginning,  we attempted to prove Theorem~\ref{thm:bijd} by using generating function but failed. To find such a proof remains an interesting open problem. 
\end{remark}

The construction of $\phi_d$ is a $d$-extension of $\phi$ defined in Theorem~\ref{thm:bij}. 
\begin{proof}[{\bf Proof of Theorem~\ref{thm:bijd}}]
Under the  correspondence $\la\mapsto w=w(\la)$ in Lemma~\ref{lem:01sequence}, the two statistics $\dif_d(\la)$ and $\mod'_d(\la)$ are transformed to  
\begin{align*}
\widetilde\dif_d(w) &= |\{ i\in[n-1]: w_i = w_{i+k} = 1\text{ and $w_{i+1}=\cdots=w_{i+k-1}=0$ for some $k\in[d]$} \}|,\\
\widetilde\mod'(w) &= |\{ 2\leq i \leq n: w_i  = 1\text{ and } \#\{0\leq j<i:w_j=0\}\not\equiv1(\mod\,d+1) \}|.
\end{align*}
So it is sufficient to define recursively a bijection $\tilde\phi_d: \Lambda_n\rightarrow \Lambda_n$ such that 
\begin{equation}\label{eq:statd}
\widetilde\dif_d(w)\geq\widetilde\mod'_d(\tilde\phi_d(w))
\end{equation}
for each $w\in\Lambda_n$ and then set $\phi_d=w^{-1}\circ\tilde\phi_d\circ w$. 

Set $\tilde\phi_d(\emptyset)=\emptyset$ and $\tilde\phi_d(01)=01$. For each $w\in\Lambda_n$ ($n\geq2$), we distinguish the following two cases:
\begin{itemize}
\item If $w=00w_2\cdots w_n$, then $\tilde\phi_d(w)$ is obtained from $\tilde\phi_d(0w_2\cdots w_n)$ by inserting  $1$ immediately after the initial $0$.
\item   If $w=01w_2\cdots w_n$, then suppose $w_m=0$ is the $(d+1)$-th $0$ in $w$ if exists, otherwise set $m=n+1$. Define $\tilde\phi_d(w)$ to be the concatenation  $00w_2\cdots w_{m-1}\tilde\phi_d(w_m\cdots w_n)$.
\end{itemize}
For example, we have
\begin{align*}
&001\mapsto011\mapsto001;\,\, 0101\mapsto 0001\mapsto 0111\mapsto 0011\mapsto 0101;\\
&00001\mapsto01111\mapsto00111\mapsto01011\mapsto00011\mapsto 01101\mapsto00101\mapsto01001\mapsto00001.
\end{align*}
under the mapping $\tilde\phi_2$.
We need to verify that $\tilde\phi_d$ satisfies~\eqref{eq:statd}  by induction on $n$ depending on two cases of $\tilde\phi_d$:
\begin{itemize}
\item Since inserting  $1$  after the initial $0$ of $\tilde\phi_d(0w_2\cdots w_n)$ dose not change $\widetilde\mod'$, we have 
$$
\widetilde\dif_d(w)=\widetilde\dif_d(0w_2\cdots w_n)\geq \widetilde\mod'_d(\tilde\phi_d(0w_2\cdots w_n))=\widetilde\mod'_d(\tilde\phi_d(w)).
$$
\item As the two statistics $\widetilde\dif_d$ and $\widetilde\mod'_d$ can be extended to any $01$ sequence and notice that there are exactly $d+1$ $0$'s in the prefix $00w_2\cdots w_{m-1}$ of $\tilde\phi_d(w)$ whenever $m\neq n+1$, we have 
\begin{align*}
\widetilde\dif_d(w)&\geq \widetilde\dif_d(01w_2\cdots w_{m-1})+\widetilde\dif_d(w_m\cdots w_{n})\\
&\geq\widetilde\mod'_d(00w_2\cdots w_{m-1})+\widetilde\mod'_d(\tilde\phi_d(w_m\cdots w_n))\\
&=\widetilde\mod'_d(\phi_d(w)).
\end{align*}
For the case $m=n+1$, we have 
$$\widetilde\dif_d(w)=\widetilde\dif_d(01w_2\cdots w_n)=\widetilde\mod'_d(00w_2\cdots w_n)=\widetilde\mod'_d(\tilde\phi_d(w)).$$
\end{itemize}

To see that $\tilde\phi_d$ is a bijection, we construct  its inverse  $\tilde\phi_d^{-1}$ explicitly. Given $w\in\Lambda_n$ ($n\geq2$), we consider the following two cases:
\begin{itemize}
\item If $w=01w_2\cdots w_n$, then $\tilde\phi_d^{-1}(w)$ is constructed from $\tilde\phi_d^{-1}(0w_2\cdots w_n)$ by inserting  $0$ immediately after the initial $0$.
\item   If $w=00w_2\cdots w_n$, then suppose $w_m=0$ is the $(d+2)$-th $0$ in $w$ if exists, otherwise set $m=n+1$. Define $\tilde\phi_d^{-1}(w)$ to be the concatenation  $01w_2\cdots w_{m-1}\tilde\phi_d^{-1}(w_m\cdots w_n)$.
\end{itemize}
It can be checked routinely that $\tilde\phi_d$ and $\tilde\phi_d^{-1}$ are inverse to each other, which proves that $\tilde\phi_d$ is a bijection. 

Finally, we need to verify that whenever $\mod'_d(w)=0$ then $\dif_d(\tilde\phi_d^{-1}(w))=0$ by induction on $n$ according to the two cases of $\tilde\phi_d^{-1}$:
\begin{itemize}
\item In the first case, we have 
$$
0=\mod'_d(w)=\mod'_d(0w_2\cdots w_n)=\dif_d(\tilde\phi_d^{-1}(0w_2\cdots w_n))=\dif_d(\tilde\phi_d^{-1}(w)).
$$
\item In the second case, since $\mod'_d(w)=0$, we must have $m=d+2$ and the prefix $00w_2\cdots w_m$ of $w$ are all $0$'s, which leads to 
$$
\dif_d(\tilde\phi_d^{-1}(w))=\tilde\phi_d^{-1}(w_m\cdots w_n)=0.
$$
\end{itemize}
The proof of the theorem is now complete. 
\end{proof}

\section{Further remarks}
\label{sec:4}

We will conclude this paper with the following three remarks.
\begin{enumerate}
\item Note that Eq.~\eqref{eq:difd1} is equivalent to the following closed formula 
$$
  \sum_{\la\in \H_n} \dif_d(\la) = (n-1)\cdot 2^{n-2}-(n-d-1)\cdot 2^{n-d-2},
$$
which can also de deduced directly as follows. By the  01-sequence representation of partitions, we have 
\begin{align*}
  &\quad\sum_{\la\in \H_n} \dif_d(\la)=\sum_{\la\in \Lambda_n} \widetilde\dif_d(\la)\\
   &= \sum_{j=1}^{d} \sum_{w\in \Lambda_n}  |\{ 1\leq i  \leq n-j: w_i = w_{i+j} = 0, ~~ w_{i+1} = w_{i+2}=\ldots=w_{i+j-1} = 1 \}| 
  \\&= 
  \sum_{j=1}^{d} \sum_{i=1}^{n-j}   |\{ w\in \Lambda_n : w_i = w_{i+j} = 0, ~~ w_{i+1} = w_{i+2}=\ldots=w_{i+j-1} = 1 \}| 
  \\&= \sum_{j=1}^{d} \left( (n-j-1)2^{n-j-2} + 2^{n-j-1} \right)
  \\&=  (n-1)\cdot 2^{n-2}-(n-d-1)\cdot 2^{n-d-2}.
\end{align*}
\item In Theorem~\ref{gf:rep-even}, we have computed the generating function for the joint distribution of $(\rep,\even)$ over $\H_n$, which is a rational formal power series. The pair $(\dif_d,\mod'_d)$ is a $d$-extension of  $(\rep,\even)$, however, in Lemmas~\ref{le:gfm} and~\ref{le:gfd} we have to deal separably with the distribution of $\dif_d$ and $\mod'_d$ on $\H_n$. 
It remains an open problem to compute the  generating function for the joint distribution of $(\dif_d,\mod'_d)$ over $\H_n$  for general $d$.
\item In~\cite{Wilf}, Wilf proved via the Principle of Inclusion-Exclusion~\cite[Sec.~2.1]{St} an interesting refinement of Euler's Odd-Distinct partition theorem by using two valued version of our (position) statistics ``$\rep$'' and ``$\even$'':
$$
\rep^*(\la)= | \{ \la_i: \la_i = \la_{i+1} \} |\quad\text{and}\quad\even^*(\la)= | \{ \la_i: \la_i \text{ is even} \} |,
$$
called the {\em number of  repeated part sizes} and the  {\em number of even part sizes}, respectively. 
Namely, he proved the equidistribution 
\begin{equation}\label{eq:wilf}
\sum_{\la\in\pa_n}t^{\rep^*(\la)}=\sum_{\la\in\pa_n}t^{\even^*(\la)},
\end{equation}
where $\pa_n$ is the set of all partitions of $n$. Interestingly, this equidistribution holds true when replacing $\pa_n$ by $\H_n$:
\begin{equation}\label{eq:lxy2}
\sum_{\la\in\H_n}t^{\rep^*(\la)}=\sum_{\la\in\H_n}t^{\even^*(\la)},
\end{equation}
which is a valued version of Theorem~\ref{th:multi_even}. See Fig.~\ref{par:peri2} for an example of~\eqref{eq:lxy2} for $n=5$. The approach of Wilf via Inclusion-Exclusion doesn't seem to work for~\eqref{eq:lxy2}. 
\begin{figure}
\begin{center}
\begin{tikzpicture}[scale=.35]

\draw(2.5,12) node{$(0,0)$};
\draw[step=1,color=gray] (0,10) grid (5,11);

\draw(9,12) node{$(0,1)$};
\draw[step=1,color=gray] (7,10) grid (11,11); 
\draw[step=1,color=gray] (7,9) grid (8,10); 

\draw(15,12) node{$(0,2)$};
\draw[step=1,color=gray] (13,10) grid (17,11); 
\draw[step=1,color=gray] (13,9) grid (15,10); 

\draw(21,12) node{$(0,1)$};
\draw[step=1,color=gray] (19,10) grid (23,11); 
\draw[step=1,color=gray] (19,9) grid (22,10); 

\draw(27,12) node{$(1,1)$};
\draw[step=1,color=gray] (25,9) grid (29,11); 

\draw(32.5,12) node{$(1,0)$};
\draw[step=1,color=gray] (31,10) grid (34,11); 
\draw[step=1,color=gray] (31,8) grid (32,10); 

\draw(37.5,12) node{$(0,1)$};
\draw[step=1,color=gray] (36,10) grid (39,11); 
\draw[step=1,color=gray] (36,8) grid (37,10); 
\draw[step=1,color=gray] (37,9) grid (38,10); 

\draw(1.5,7) node{$(1,1)$};
\draw[step=1,color=gray] (0,5) grid (3,6); 
\draw[step=1,color=gray] (0,3) grid (2,5); 

\draw(6.5,7) node{$(1,0)$};
\draw[step=1,color=gray] (5,4) grid (8,6); 
\draw[step=1,color=gray] (5,3) grid (6,4); 

\draw(11.5,7) node{$(1,1)$};
\draw[step=1,color=gray] (10,4) grid (13,6); 
\draw[step=1,color=gray] (10,3) grid (12,4); 

\draw(16.5,7) node{$(1,0)$};
\draw[step=1,color=gray] (15,3) grid (18,6); 

\draw(21,7) node{$(1,1)$};
\draw[step=1,color=gray] (20,5) grid (22,6); 
\draw[step=1,color=gray] (20,2) grid (21,5); 

\draw(25,7) node{$(2,1)$};
\draw[step=1,color=gray] (24,4) grid (26,6); 
\draw[step=1,color=gray] (24,2) grid (25,4); 

\draw(29,7) node{$(1,1)$};
\draw[step=1,color=gray] (28,3) grid (30,6); 
\draw[step=1,color=gray] (28,2) grid (29,3); 

\draw(33,7) node{$(1,1)$};
\draw[step=1,color=gray] (32,2) grid (34,6); 

\draw(36.5,7) node{$(1,0)$};
\draw[step=1,color=gray] (36,1) grid (37,6);

\end{tikzpicture}
\end{center}
\caption{Partitions in $\H_5$ (represented as Young diagrams) with their pair of statistics $(\rep^*,\even^*)$ on the top.\label{par:peri2}}
\end{figure}
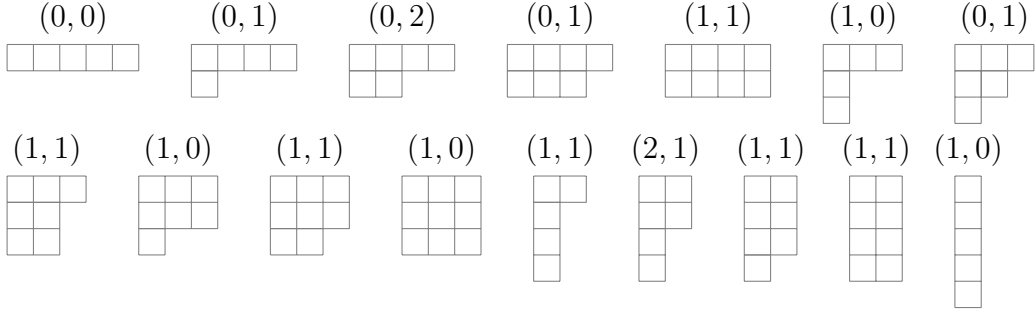
However, the bijection $\phi$ constructed in Theorem~\ref{thm:bij} also proves~\eqref{eq:lxy2}. Is there any unified generalization of both~\eqref{eq:wilf} and~\eqref{eq:lxy2}? 
\end{enumerate}

\section*{Acknowledgments}
 Authors  thank Grimaldi for kindly sending them a version of his paper~\cite{Gri}. 
 This work was supported by the National Science Foundation of China grants  11871247 and 12071440, and the project of Qilu Young Scholars of Shandong University.

 \end{document}